\documentclass[12pt]{amsart}
\usepackage[utf8]{inputenc}
\usepackage{amssymb}
\usepackage{hyperref}
\usepackage[final]{showkeys}

\input xy
\xyoption{all}

\theoremstyle{definition}
\newtheorem{mydef}{Definition}[section]
\newtheorem{lem}[mydef]{Lemma}
\newtheorem{thm}[mydef]{Theorem}
\newtheorem{conjecture}[mydef]{Conjecture}

\newtheorem{claim}[mydef]{Claim}

\newtheorem{prop}[mydef]{Proposition}
\newtheorem{defin}[mydef]{Definition}

\newtheorem{remark}[mydef]{Remark}

\newtheorem{fact}[mydef]{Fact}

\newcommand{\fct}[2]{{}^{#1}#2}

\newcommand{\ba}{\bar{a}}

\newcommand{\cf}[1]{\text{cf} (#1)}
\newcommand{\seq}[1]{\langle #1 \rangle}
\newcommand{\rest}{\upharpoonright}

\newcommand{\Mod}{\text{Mod}}

\def\lea{\le}

\newcommand{\K}{K}

\def\lee{\preceq}

\newbox\noforkbox \newdimen\forklinewidth
\forklinewidth=0.3pt \setbox0\hbox{$\textstyle\smile$}
\setbox1\hbox to \wd0{\hfil\vrule width \forklinewidth depth-2pt
 height 10pt \hfil}
\wd1=0 cm \setbox\noforkbox\hbox{\lower 2pt\box1\lower
2pt\box0\relax}
\def\unionstick{\mathop{\copy\noforkbox}\limits}

\def\1nf{\unionstick^{(1)}}

\def\2nf{\unionstick^{(2)}}
\def\3nf{\unionstick^{(3)}}

\newcommand{\EM}{\operatorname{EM}}

\newcommand{\LS}{\text{LS}}

\title{Categoricity and infinitary logics}
\date{\today}

\author{Will Boney}
\email{wboney@math.harvard.edu}
\urladdr{http://math.harvard.edu/\textasciitilde wboney/}
\address{Department of Mathematics, Harvard University, Cambridge, Massachusetts, USA}
\thanks{This material is based upon work done while the first author was supported by the National Science Foundation under Grant No. DMS-1402191.}

\author{Sebastien Vasey}
\email{sebv@cmu.edu}
\urladdr{http://math.cmu.edu/\textasciitilde svasey/}
\address{Department of Mathematical Sciences, Carnegie Mellon University, Pittsburgh, Pennsylvania, USA}
\thanks{The second author is supported by the Swiss National Science Foundation.}

\begin{document}

\begin{abstract}
  We point out a gap in Shelah's proof of the following result:

  \begin{claim}
    Let $K$ be an abstract elementary class categorical in unboundedly many cardinals. Then there exists a cardinal $\lambda$ such that whenever $M, N \in K$ have size at least $\lambda$, $M \lea N$ if and only if $M \lee_{L_{\infty, \LS (K)^+}} N$.
  \end{claim}

  The importance of the claim lies in the following theorem, implicit in Shelah's work:

  \begin{thm}\label{main-thm-2}
    Assume the claim. Let $K$ be an abstract elementary class categorical in unboundedly many cardinals. Then the class of $\lambda$ such that:
    \begin{enumerate}
      \item $K$ is categorical in $\lambda$;
      \item $K$ has amalgamation in $\lambda$; and
      \item there is a good $\lambda$-frame with underlying class $K_\lambda$
    \end{enumerate}
    is stationary.
  \end{thm}

  We give a proof and discuss some related questions.
\end{abstract}

\maketitle

\tableofcontents
\section{Introduction}

A major driving force in the field of classification theory for non-elementary classes is Shelah's categoricity conjecture\footnote{For a review of the literature on the conjecture, see the introduction of \cite{ap-universal-v6}.}:

\begin{conjecture}[Open problem D.(3a) in \cite{shelahfobook}]\label{categ-conj}
  If $L$ is a countable language and $\psi \in L_{\omega_1, \omega}$ is categorical in one $\mu \ge \beth_{\omega_1}$, then it is categorical in all $\mu \ge \beth_{\omega_1}$.
\end{conjecture}

In the framework of abstract elementary classes (AECs), Shelah states an eventual version of the conjecture as:

\begin{conjecture}[Conjecture N.4.2 in \cite{shelahaecbook}]\label{shelah-eventual-categ}
  An AEC that is categorical in a high-enough cardinal is categorical on a tail of cardinals.
\end{conjecture}

More precisely, there should exist a function $\lambda \mapsto \mu_\lambda$ such that if $K$ is an AEC categorical in \emph{some} $\mu \ge \mu_{\LS (K)}$, then it is categorical in \emph{all} $\mu' \ge \mu_{\LS (K)}$. There are several conjectures as to what the above function should be. Hart and Shelah \cite{hs-example} have shown that, as opposed to the first-order setup, $\mu_{\aleph_0} \ge \aleph_{\omega}$ and Shelah conjectures that this is essentially optimal in Conjecture N.4.3 of \cite{shelahaecbook}. Still he says it is probably more realistic to expect to prove $\mu_{\lambda} \le \beth_{(2^{\lambda})^+}$. 

However if one is only interested in the \emph{existence} of the map $\lambda \mapsto \mu_\lambda$, then a Hanf number argument (using the axiom of replacement, see \cite[Conclusion 15.13]{baldwinbook09}) shows that categoricity in a high-enough cardinal implies categoricity in unboundedly many cardinals. Therefore for the purpose of establishing Conjecture \ref{shelah-eventual-categ} (sometimes also called Shelah's \emph{eventual} categoricity conjecture), we can restrict ourselves to AECs categorical in unboundedly many cardinals\footnote{One can of course also ask what the Hanf number for unbounded categoricity is. That is, given $\lambda$, what is the least $\mu_0$ such that an AEC with Löwenheim-Skolem number $\lambda$ categorical in some $\mu \ge \mu_0$ is categorical in a unboundedly many cardinals. Even assuming large cardinals, we are not aware of any known explicit bound.}. In such AECs, several difficulties with dealing with categoricity in a single cardinal disappear. For example, assuming amalgamation and no maximal models, a major issue is to show that a class categorical in a $\lambda > \LS (K)$ has a Galois-saturated model in $\lambda$. This is one reason why Shelah assumes that $\lambda$ is regular in \cite{sh394}. However if $K$ is also categorical in a $\lambda' > \lambda$, then $K$ will be Galois-stable in $\lambda$, hence the model of size $\lambda$ will automatically be Galois-saturated. A consequence of this result is:

\begin{prop}\label{easy-ap}
  Assume that $K$ is an AEC with amalgamation and no maximal models, categorical in unboundedly many cardinals. Let $\lambda > \LS (K)$ be such that $K$ is categorical in $\lambda$. For any $M, N \in K_{\lambda}$, if $M \lea N$, then $M \lee_{L_{\infty, \LS (K)^+}} N$.
\end{prop}
\begin{proof}[Proof sketch]
  By \cite[Claim 1.7]{sh394}, $K$ is Galois-stable in $\lambda$. Thus it is easy to see that both $M$ and $N$ are Galois-saturated. Now let $\phi$ be an $L_{\infty, \LS (K)^+}$-formula and let $\ba \in \fct{\LS (K)}{|M|}$. Let $M_0 \lea M$ have size $\LS (K)$ and contain $\ba$. By uniqueness of saturated models, there exists $f: N \cong_{M_0} M$. Therefore $M \models \phi[\ba]$ if and only if $N \models \phi[\ba]$.
\end{proof}

The proof of Proposition \ref{easy-ap} uses amalgamation and no maximal models heavily, for example to obtain that $M$ and $N$ are Galois-saturated and to build the isomorphism taking $M$ to $N$ and fixing $M_0$. 

In Chapter IV of \cite{shelahaecbook}, Shelah discusses variations of Proposition \ref{easy-ap} which do \emph{not} assume amalgamation and no maximal models. The change to make is to take $\lambda$ much bigger than $\LS (K)^+$. Shelah shows: 

\begin{fact}\label{syntactic-facts}
  Let $K$ be an AEC and let $\theta$ be an infinite cardinal.

  \begin{enumerate}
    \item \cite[Claim IV.1.12.(1)]{shelahaecbook} If\footnote{Shelah assumes in addition that $\theta > \LS (K)$, but the proof shows that it is not necessary.} $K$ is categorical in a $\lambda = \lambda^{<\theta}$, then for any $M, N \in K_{\ge \lambda}$, $M \lea N$ implies $M \lee_{L_{\infty, \theta}} N$.
    \item \cite[Conclusion IV.2.12.(1)]{shelahaecbook} If $\theta > \LS (K)$ and $K$ is categorical in a $\lambda > 2^{<\theta}$ with $\cf{\lambda} \ge \theta$, then whenever $M, N \in K_{\ge \lambda}$, $M \lea N$ implies $M \lee_{L_{\infty, \theta}} N$.
  \end{enumerate}
\end{fact}

In fact, Shelah even claims \cite[Conclusion IV.2.14]{shelahaecbook}:

\begin{claim}\label{shelah-claim-0}
  Let $K$ be an AEC categorical in unboundedly many cardinals. For any $\theta$, there exists a cardinal $\mu_0 (\theta)$ such that for any $\mu$ so that $K$ is categorical in $\mu \ge \mu_0 (\theta)$, for any $M, N \in K_{\ge \mu}$, $M \lea N$ implies $M \lee_{L_{\infty, \theta}} N$.
\end{claim}

Claim \ref{shelah-claim-0} has several applications. For one thing, we can use it to show that for an AEC $K$ categorical in unboundedly many cardinals, there exists $\lambda$ so that $K_{\ge \lambda}$ is nothing but the class of models of a complete $L_{(2^{\LS (K)})^+, \LS (K)^+}$-sentence, ordered by $L_{\infty, \LS (K)^+}$-elementary substructure. This adds to a result of Kueker (\cite[Theorem 7.4]{kueker2008}) which proved this (without the characterization of the ordering) assuming amalgamation and categoricity in a single cardinal of high-enough cofinality\footnote{Kueker also shows \cite[Theorem 7.2]{kueker2008} that $M \lee_{\infty, \LS (K)^+} N$ implies $M \lea N$, but we are interested in the converse here.} .

More importantly, the condition that $M \lea N$ implies $M \lee_{L_{\infty, \theta}} N$ (for an appropriate $\theta$) is crucial in Chapter IV of \cite{shelahaecbook}. There Shelah uses the condition to obtain amalgamation in a single cardinal from categoricity in a suitable cardinal (see \cite[Theorem IV.1.30]{shelahaecbook}). An implicit conclusion of this work (the second part appeared explicitly in earlier versions of \cite{ap-universal-v6}) is:

\begin{claim}\label{shelah-claim}
  If $K$ is an AEC categorical in unboundedly many cardinal, then:

  \begin{enumerate}
    \item $K$ is categorical in a stationary class of cardinals.
    \item There exists a categoricity cardinal $\lambda$ such that $K$ has amalgamation in $\lambda$. In fact, $K$ has a type-full good $\lambda$-frame in $\lambda$ (see \cite[Definition II.2.1]{shelahaecbook}).
  \end{enumerate}
\end{claim}

The second part of claim \ref{shelah-claim} has been used by the second author in an early version of his proof of Shelah's eventual categoricity conjecture for universal classes \cite{ap-universal-v6}\footnote{We believe that such claims, obtaining some amalgamation from categoricity, will be central in the resolution of Shelah's categoricity conjecture. In fact, it has been conjectured by Grossberg (see \cite[Conjecture 2.3]{grossberg2002}) that amalgamation should simply follow from categoricity in a high-enough cardinal.}. 

\emph{However}, we have identified a gap in Shelah's proof of Claim \ref{shelah-claim-0} (hence invalidating Claim \ref{shelah-claim}): in \cite[Section IV.2]{shelahaecbook}, Shelah uses model-theoretic forcing and proves \cite[Claim IV.2.13]{shelahaecbook} that it is enough to have a certain downward Löwenheim-Skolem theorem for forcing. Shelah had previously established a downward Löwenheim-Skolem theorem \emph{for $L_{\infty, \theta}$}, and claims without comments (in his proof of Conclusion IV.2.14) that it suffices to use it. However it is not clear that forcing and the classical $L_{\infty, \theta}$ satisfaction relation coincide, in fact this is essentially what Shelah wants to prove.

Here, we explain why Claim \ref{shelah-claim} follows from Claim \ref{shelah-claim-0} and discuss some related questions on the interaction between logics such as $L_{\infty, \theta}$ and AECs (thus this paper can be considered a follow up to Kueker's work \cite{kueker2008})\footnote{In an earlier version of this paper, we claimed to fix Shelah's gap and prove Claim \ref{shelah-claim-0}, but an error was later found: Lemma 2.20 in versions dated prior to Oct. 15, 2015 is false.}.

This paper was written while the second author was working on a Ph.D.\ thesis under the direction of Rami Grossberg at Carnegie Mellon University and he would like to thank Professor Grossberg for his guidance and assistance in his research in general and in this work specifically. We also thank John Baldwin for comments on an early version of this paper.

\section{Syntactically characterizable AECs}\label{syntactic-sec}

To make it easy to speak of the conclusion of Claim \ref{shelah-claim-0}, we give it a name. 

\begin{defin}
  An AEC $\K$ is \emph{$L_{\infty, \theta}$-syntactically characterizable} if whenever $M, N \in \K$, if $M \lea N$ then $M \lee_{L_{\infty, \theta}} N$. We say that $\K$ is \emph{eventually syntactically characterizable} if for every infinite cardinal $\theta$, there exists $\lambda$ such that $\K_{\ge \lambda}$ is $L_{\infty, \theta}$-syntactically characterizable.
\end{defin}

Thus Claim \ref{shelah-claim-0} says that if $K$ is categorical in unboundedly many cardinals, then $K$ is eventually syntactically characterizable. The terminology is justified by the following lemma:

\begin{lem}\label{event-synt-lem}
  Assume that $\K$ is an AEC with $\kappa := \LS (\K)$. If $K$ is $L_{\infty, \kappa^+}$-syntactically characterizable, then there is a formula $\phi \in L_{(2^{\kappa})^+, \kappa^+}$ such that $(K, \lea) = (\Mod (\phi), \lee_{L_{\infty, \kappa^+}})$. Moreover if $K_{<\kappa} = \emptyset$ and $K$ is categorical in $\kappa$, then $\phi$ can be taken to be complete.
\end{lem}
\begin{proof}
  By Shelah \cite[Claim IV.1.10.(4)]{shelahaecbook} or Kueker \cite[Theorem 7.2.(a)]{kueker2008}), we have that if $M \in K$ and $M \equiv_{L_{\infty, \kappa^+}} N$, then $N \in K$ (this does not need syntactic characterizability). For for each $M \in \K_{\le \kappa}$, let $T_M$ be the complete $L_{\infty, \kappa^+}$-theory of $M$, and let $\phi_M \in L_{(2^{\kappa})^+, \kappa^+}$ code it (as in the proof of Scott's isomorphism theorem, see \cite[Corollary 5.3.33]{dickmann-book} or \cite[Claim IV.2.8.(2)]{shelahaecbook}). If $N \in K$, find $M \lea N$ with $M \in K_{\le \kappa}$. We know that $M \lee_{L_{\infty, \kappa^+}} N$ and $M \models T_{M}$ so also $N \models T_{M}$. Therefore we can take $\phi := \bigvee_{M \in \K_{\le \kappa}} \phi_M$.

  Finally (see Shelah \cite[Claim IV.1.10.(3)]{shelahaecbook} or Kueker \cite[Theorem 7.2.(b)]{kueker2008}), if $M \lee_{L_{\infty, \kappa^+}} N$ and $N \in K$, then $M \in K$ and $M \lea N$. This completes the proof.
\end{proof}

For the rest of this section, we give some properties of categorical syntactically characterizable AECs. In \cite[Theorem 7.4]{kueker2008}, Kueker proved:

\begin{fact}
  Let $K$ be an AEC with amalgamation, joint embedding, and no maximal models. Let $\kappa := \LS (K)$. Let $\lambda \ge \LS (K)$ be a cardinal such that $\cf{\lambda} > \kappa$. Then there is a complete $\phi \in L_{(2^{\kappa})^+, \kappa^+}$ such that $K_{\ge \lambda} = (\Mod (\phi))_{\ge \lambda}$.
\end{fact}

Using Fact \ref{syntactic-facts}, we can remove amalgamation, joint embedding, and no maximal models from Kueker's result (assuming in addition that $\lambda > 2^{\LS (\K)}$). Moreover we can say what the ordering on $K$ looks like:

\begin{thm}\label{kueker-improvement}
  Let $K$ be an AEC and let $\kappa := \LS (K)$. If $K$ is categorical in a $\lambda > 2^\kappa$ with $\cf{\lambda} > \kappa$, then there exists a complete $\phi \in L_{(2^{\kappa})^+, \kappa^+}$ such that $K_{\ge \lambda} = (\text{Mod} (\phi))_{\ge \lambda}$ and for $M, N \in K_{\ge \lambda}$, $M \lea N$ if and only if $M \lee_{L_{\infty, \kappa^+}} N$.
\end{thm}
\begin{proof}
  By Fact \ref{syntactic-facts}, $\K_{\ge \lambda}$ is $L_{\infty, \kappa^+}$-syntactically characterizable. Now apply Lemma \ref{event-synt-lem}.
\end{proof}

\subsection{Shelah's amalgamation theorem}

In Chapter IV of his book on AECs \cite{shelahaecbook}, Shelah proves, assuming eventual syntactic characterizability, that categoricity in a high-enough cardinal implies amalgamation in a specific cardinal below the categoricity cardinal. For what follows, we assume some familiarity with Ehrenfeucht-Mostowski models, see for example \cite[Section 6.2]{baldwinbook09} or \cite[Definition IV.0.8]{shelahaecbook}. Recall: 

\begin{defin}
  Let $K$ be an AEC. We say that $\Phi$ is an \emph{EM blueprint for $K$} if $\Phi$ is a set of quantifier-free types proper for linear orders in a vocabulary $\tau (\Phi) \supseteq \tau (K)$, and (writing $\EM (I, \Phi)$ for the $\tau (\Phi)$-structure generated by the linear order $I$ and $\Phi$ and $\EM_{\tau (K)} (I, \Phi)$ for $\EM (I, \Phi) \rest \tau (K)$):

  \begin{enumerate}
    \item For any linear order $I$, $\EM_{\tau (K)} (I, \Phi) \in K$ and $\|\EM (I, \Phi)\| = |I| + |\tau'| + \LS (K)$.
    \item For any linear orders $I$, $J$, if $I \subseteq J$, then $\EM_{\tau (K)} (I, \Phi) \lea \EM_{\tau (K)} (J, \Phi)$.
  \end{enumerate}
\end{defin}

From the presentation theorem, we get (see e.g.\ \cite[Claim 0.6]{sh394}):

\begin{fact}
  If an AEC $K$ has arbitrarily large models, then there exists an EM blueprint $\Phi$ for $K$ with $|\tau (\Phi)| = \LS (K)$. 
\end{fact}

\begin{defin}\label{k-ast-def}
  Let $K$ be an AEC and let $\Phi$ be an EM blueprint for $K$. We write $K_{\Phi}^\ast$ (or just $K^\ast$ when $\Phi$ is clear from context) for the class of $M \in K$ such that there exists a linear order $I$ with $M \cong \EM_{\tau (K)} (I, \Phi)$.
\end{defin}

If in addition $K$ is categorical in $\lambda \ge \LS (K)$, then in particular all the $\EM$-models of size $\lambda$ must be isomorphic. In fact, the class of $\EM$-models of size $\lambda$ will generate an AEC. This property is isolated by Shelah in \cite[Chapter IV]{shelahaecbook} and called \emph{solvability}. We define only the weaker notion of pseudo-solvability here (regular solvability asks in addition that the $\EM$-model of size $\lambda$ be superlimit, so in particular every $M \in K_\lambda$ embeds into an $\EM$-model of size $\lambda$).

\begin{defin}[Definition IV.1.4.(3) in \cite{shelahaecbook}]\label{solvability-def}
  Let $K$ be an AEC and let $\Phi$ be an EM blueprint for $K$. We say that $(K, \Phi)$ is \emph{pseudo $(\lambda, \theta)$-solvable} if:

  \begin{enumerate}
    \item $\LS (K) \le \theta \le \lambda$. $|\Phi| \le \theta$.
    \item $K_{\Phi}^\ast$ is categorical in $\lambda$.
    \item If $\delta < \lambda^+$ and $\seq{M_i : i < \delta}$ are increasing in $K_{\Phi}^\ast$ and have size $\lambda$, then $\bigcup_{i < \delta} M_i$ is in $K_{\Phi}^\ast$.
  \end{enumerate}

  We say that $K$ is \emph{pseudo $(\lambda, \theta)$-solvable} if $(K, \Phi)$ is pseudo $(\lambda, \theta)$-solvable for some EM blueprint $\Phi$.
\end{defin}
\begin{remark}
  If $K$ has arbitrarily large models, $\Phi$ is an EM blueprint for $K$, and $K$ is categorical in a $\lambda \ge \LS (K) + |\tau (\Phi)|$, then $(K, \Phi)$ is pseudo $(\lambda, |\tau (\Phi)| + \LS (K))$-solvable.
\end{remark}

\begin{fact}[Shelah's amalgamation theorem]\label{shelah-ap}
  Let $K$ be an AEC and let $\lambda, \mu$ be cardinals such that $\LS (K) < \lambda = \beth_\lambda < \mu$ and $\cf{\lambda} = \aleph_0$. Assume that there exists an EM blueprint $\Phi$ such that $(K, \Phi)$ is pseudo $(\mu, \LS (K))$-solvable.

  Assume further that for every $\theta \in [\LS (K), \lambda)$, and every $M, N \in K_{\Phi}^\ast$ (see Definition \ref{k-ast-def}) of size $\mu$, $M \lea N$ implies $M \lee_{L_{\infty, \theta}} N$. Then:

    \begin{enumerate}
      \item (\cite[Theorem IV.1.30]{shelahaecbook}) $K_{\Phi}^\ast$ has disjoint amalgamation in $\lambda$.
      \item (\cite[Theorem IV.4.10]{shelahaecbook}) If in addition there exists an increasing sequence of cardinals $\seq{\lambda_n : n < \omega}$ such that:
        \begin{enumerate}
          \item $\sup_{n < \omega} \lambda_n = \lambda$.
          \item For all $n < \omega$, $\lambda_n = \beth_{\lambda_n}$ and $\cf{\lambda_n} = \aleph_0$.
        \end{enumerate}

        Then there is a type-full good $\lambda$-frame with underlying class $(K_{\Phi}^\ast)_{\lambda}$.
    \end{enumerate}
\end{fact}

With similar methods, Shelah also proves a categoricity transfer. This is given by (the proof of) \cite[Observation IV.3.5]{shelahaecbook}:

\begin{fact}\label{shelah-categ-transfer}
  Let $K$ be an AEC and let $\seq{\lambda_n : n \le \omega}$ be an increasing continuous sequence of cardinals such that for all $n < \omega$:

  \begin{enumerate}
  \item $\LS (K) < \lambda_{\omega} = \beth_{\lambda_{\omega}}$.
  \item $K$ is categorical in $\lambda_n$.
  \item For every $\theta < \lambda_n$, $K_{\ge \lambda_{n + 1}}$ is $L_{\infty, \theta}$-syntactically characterizable.
  \end{enumerate}

  Then $K$ is categorical in $\lambda_{\omega}$.
\end{fact}

Combining these two facts with Theorem \ref{kueker-improvement}, we obtain Theorem \ref{main-thm-2} from the abstract. The argument already appeared in an early version of \cite{ap-universal-v6}:

\begin{thm}
  Let $K$ be an AEC categorical in unboundedly many cardinals. Assume that $K$ is eventually syntactically characterizable. Set $S$ to be the class of cardinals $\lambda$ such that:

  \begin{enumerate}
    \item $\LS (K) < \lambda = \beth_\lambda$ and $\cf{\lambda} = \aleph_0$;
    \item $K$ is categorical in $\lambda$; and
    \item there is a type-full good $\lambda$-frame with underlying class $K_{\lambda}$ (in particular, $K$ has amalgamation in $\lambda$).
  \end{enumerate}

  Then $S$ is stationary.
\end{thm}
\begin{proof}
  For each cardinal $\theta$, let $\mu_0 (\theta)$ be the least cardinal $\mu > \theta$ such that $K$ is categorical in $\mu$ and for each $M, N \in K_{\ge \mu}$, $M \lea N$ implies $M \lee_{L_{\infty, \theta}} N$. Note that $\mu_0 (\theta)$ exists by definition of eventual syntactic characterizability.

  Let $C$ be a closed unbounded class of cardinals.

  Build $\seq{\lambda_i : i \le \omega \cdot \omega}$ and $\seq{\lambda_i' : i \le \omega \cdot \omega}$ increasing continuous such that for all $i \le \omega \cdot \omega$:

  \begin{enumerate}
    \item $\lambda_i' \in C$.
    \item For all $j < i$, $\lambda_j < \lambda_i' \le \lambda_i$.
    \item $\lambda_0 > \LS (K)$
    \item $K$ is categorical in $\lambda_i$
    \item $\lambda_{i + 1} > \mu_0(\beth_{\lambda_i})$.
  \end{enumerate}
  
  This is possible: when $i$ is a zero or a successor, we pick $\lambda_i' \in C$ such that $\lambda_j < \lambda_i'$ for all $j < i$ (this is possible as $C$ is unbounded). Then we pick a categoricity cardinal $\lambda_i \ge \lambda_i'$ strictly above $\LS (K)$ (and, if $i = j + 1$ also strictly above $\mu_0 (\lambda_j)$). For $i$ limit, let $\lambda_i := \lambda_i' := \sup_{j < i} \lambda_j$. Note that $\lambda_i' \in C$ as $C$ is closed. Also, $\beth_{\lambda_{i}} = \lambda_i$ and $\cf{\lambda_i} = \aleph_0$. Therefore by Fact \ref{shelah-categ-transfer}, $K$ is categorical in $\lambda_i$. This is enough: let $\lambda := \lambda_{\omega \cdot \omega}$. Since $C$ is club, $\lambda \in C$. Also, $\lambda$ is a limit of fixed points of the beth function of cofinality $\aleph_0$. Moreover $K$ is categorical (and hence pseudo solvable) in some $\mu \ge \mu_0 (\lambda)$, so by Fact \ref{shelah-ap} we obtain that $\lambda \in S$.
\end{proof}

\section{More on syntax in AECs}

This section is in preparation.

\bibliographystyle{amsalpha}
\bibliography{categoricity-infinitary-logics}

\end{document}